\documentclass[11pt]{amsart}
\usepackage{amsmath,amsthm,amsfonts,amscd,amssymb,eucal,latexsym,mathrsfs,stmaryrd,enumitem,bm}
\usepackage[all]{xy}
\usepackage{mathtools}
\usepackage{enumitem}  
\usepackage{comment}

\newtheorem{theorem}{Theorem}[section]

\newtheorem{lemma}[theorem]{Lemma}

\theoremstyle{definition}
\newtheorem{definition}[theorem]{Definition}
\newtheorem{remark}[theorem]{Remark}

\theoremstyle{plain}
\newcounter{theoremintro}
\newtheorem{theoremi}[theoremintro]{Theorem}

\newcommand{\cU}{{\mathcal U}}

\newcommand{\cZ}{{\mathcal Z}}

\newcommand{\sU}{{\mathscr U}}

\newcommand{\Zb}{{\mathbb Z}}

\newcommand{\eps}{\varepsilon}

\newcommand{\abs}[1]{\left|#1\right|}
\newcommand{\brak}[1]{\left(#1\right)}

\numberwithin{equation}{section}

\allowdisplaybreaks

\begin{document}

\title{Polynomial growth, comparison, and the small boundary property}

\author{Petr Naryshkin}
\address{Petr Naryshkin,
Mathematisches Institut,
WWU M{\"u}nster, 
Einsteinstr.\ 62, 
48149 M{\"u}nster, Germany}
\email{pnaryshk@uni-muenster.de}


\begin{abstract}
We show that a minimal action of a finitely generated group of polynomial growth on a compact metrizable space has comparison. It follows that if such an action is free and has the small boundary property then it is almost finite and its $C^*$-crossed product is $\cZ$-stable, and consequently that such crossed products are classified by their Elliott invariant.
\end{abstract}

\date{\today}

\maketitle

\section{Introduction}
Over the last couple of decades the notion of $\cZ$-stability has come to play a central role in the classification theory for separable simple nuclear $C^*$-algebras, with its decisive importance as a regularity property having been recently cemented in \cite{CasEviTikWhiWin18}. It has therefore become a significant problem to determine when a free minimal action $G \curvearrowright X$ of a countable amenable group on a compact metrizable space gives rise to a $\cZ$-stable $C^*$-crossed product $C(X) \rtimes G$. Using the theory of subhomogeneous algebras, Niu and Elliott proved in \cite{EllNiu17} that this holds when $G=\Zb$ and $G \curvearrowright X$ has mean dimension zero. Later, Niu improved on this result and showed (\cite{Niu19b}, \cite{Niu19a}) that when $G = \Zb^d$ the radius of comparison of the crossed product $C(X) \rtimes \Zb^d$ is at most half of the mean dimension of the action. On the other hand, Hirshberg and Phillips \cite{HirPhi20}, following the work of Giol and Kerr \cite{GioKer10}, proved that for a certain class of minimal subshifts the radius of comparison of the crossed product is bounded from below in terms of the mean dimension and under some conditions is at least $\frac{1}{2}\mathrm{mdim}(G \curvearrowright X) - 1$.

An alternative approach was developed by Kerr who introduced a dynamical notion of almost finiteness \cite{Ker20} for a free action of an amenable group on a compact metrizable space and showed that when minimal it implies the $\cZ$-stability of crossed product. Kerr and Szab{\'o} \cite{KerSza20} proved that almost finiteness is equivalent to the action having both the small boundary property and comparison. They also showed that if every action of a group $G$ on a zero-dimensional space is almost finite then every action of $G$ on a finite-dimensional space is almost finite (note that for these actions the small boundary property is automatic). There have been a number of results establishing almost finiteness when the space $X$ is zero-dimensional:  Downarowicz and Zhang proved \cite{DowZha17} that the actions of groups with locally subexponential growth have comparison; Conley--Jackson--Kerr--Marks--Seward--Tucker-Drob showed \cite{ConJacKerMarSewTuc18} that a generic action of any countable amenable group is almost finite when the space $X$ is assumed to be perfect (i.e. a Cantor set); Conley--Jackson--Marks--Seward--Tucker-Drob proved \cite{ConJacMarSewTuc20} later that actions of groups from a certain class which includes polycyclic groups as well as the lamplighter group are almost finite; finally, in a recent paper by Kerr and the author \cite{KerNar21} it was shown that actions of all elementary amenable groups are almost finite. In view of the result of Kerr and Szab{\'o} above, we thus now know for a large class of groups that their free actions on finite dimensional spaces are almost finite.

The main result of this paper is the following. 

\begin{theoremi}\label{T-A}
A minimal action $G \curvearrowright X$ of a finitely generated group of polynomial growth on a compact metrizable space has comparison.
\end{theoremi}

Note that by the celebrated theorem of Gromov \cite{Gro81} a finitely generated groups has polynomial growth if and only if it is virtually nilpotent. The next statement is a corollary to Theorem \ref{T-A} and follows from the results of Kerr \cite{Ker20} and Kerr-Szab{\'o} \cite{KerSza20} mentioned above.

\begin{theoremi}\label{T-B}
A free minimal action $G \curvearrowright X$ of a finitely generated group of polynomial growth on a compact metrizable space has the small boundary property if and only if it is almost finite. It follows that when $G \curvearrowright X$ has the small boundary property its $C^*$-crossed product $C(X) \rtimes G$ is $\cZ$-stable. Consequently, these crossed products are classified by their Elliott invariant.
\end{theoremi}

In an independent work Bosa, Perera, Wu, and Zacharias \cite{BosPerWuZac21} have used the theory of Cuntz semigroup to establish Theorem \ref{T-A} for $\Zb^d$-actions. In this setting a version of comparison and the conclusion of Theorem \ref{T-B} also appears in the work of Niu \cite{Niu19a} mentioned above (note that the small boundary property is equivalent to mean dimension zero for $\Zb^d$-actions \cite{GutLinTsu16}). Our approach, however, is more purely dynamical in nature and in particular avoids the use of subhomogeneous subalgebras. 

The structure of the paper is as follows. In section \ref{ComparisonSec} we recall the definition of comparison as well as give several weaker notions all of which end up being equivalent for minimal actions (Lemma \ref{ComparisonLemma}). In section \ref{ResultsSec} we prove Theorems \ref{T-A} and \ref{T-B} using one of these notions and give a few closing remarks.
\medskip

\noindent{\it Acknowledgements.}
The author wants to thank David Kerr for guidance and comments. The author is also grateful to Eusebio Gardella, Grigoris Kopsacheilis, and Julian Kranz for helpful comments and suggesions. The research was partially funded by the Deutsche Forschungsgemeinschaft (DFG, German Research Foundation) under Germany’s Excellence Strategy – EXC 2044 – 390685587, Mathematics Münster – Dynamics – Geometry – Structure; the Deutsche Forschungsgemeinschaft (DFG, German Research Foundation) – Project-ID 427320536 – SFB 1442, and ERC Advanced Grant 834267 - AMAREC

\section{Comparison properties}
\label{ComparisonSec}

Let $G\curvearrowright X$ be a continuous action 
of a countable amenable group on a compact metric space.

Let $A$ and $B$ be subsets of $X$. Recall that $A$ is {\it subequivalent} to $B$ ($A\prec B$) if for every closed set $C \subseteq A$ there are 
finitely many open sets $U_1 , \dots , U_n$ and elements $s_1 , \dots , s_n \in G$ such that $C 
\subset \bigcup_{i=1}^n U_i$ and the sets $s_i U_i$ for 
$i=1,\dots , n$ are pairwise disjoint and contained in $B$. 
We say that $A$ is {\it $m$-subequivalent} to $B$ ($A\prec_m B$) 
if for every closed set $C\subseteq A$ there are a finite open cover $\sU$ 
of $C$, an element $s_U \in G$ for every $U \in \sU$, 
and a partition of $\sU$ into subcollections $\sU_0 , \dots , \sU_m$ 
such that for every $i = 0,...,m$ the sets $s_U U$ for $U \in \sU_i$ 
are pairwise disjoint and contained in $B$. 

Let $A$ be a subset of $X$ and $\eps>0$. By $A^\eps$ we denote the set $\{x \in X \mid d(x, A) < \eps\}$ where $d$ is the metric on $X$. Similarly, by $A^{-\eps}$ we denote the set $\{x \in A \mid d(x, X \setminus A) > \eps\}$.

For a set $A \subset X$ we define lower and upper densities respectively:
$$\underline{D}(A) = \sup_{F \Subset G} \inf_{x \in X} \frac{1}{\abs{F}} \sum_{s \in F} 1_A(sx), \qquad \overline{D}(A) = \inf_{F \Subset G} \sup_{x \in X} \frac{1}{\abs{F}} \sum_{s \in F} 1_A(sx).$$

By \cite[Lemma 3.2]{KerSza20} if $F_n$ is a F{\o}lner sequence for $G$,
    $$\underline{D}(A) = \lim\limits_{n\to\infty} \inf_{x \in X} \frac{1}{\abs{F_n}} \sum_{s \in F_n} 1_A(sx), \qquad
    \overline{D}(A) = \lim\limits_{n\to\infty} \sup_{x \in X} \frac{1}{\abs{F_n}} \sum_{s \in F_n} 1_A(sx).$$

Let $M_G(X)$ be the space of $G$-invariant probability measures on $X$. By \cite[Propositions 3.3 and 3.4]{KerSza20},
\begin{gather}
\label{LowerDensityFormula}
    \underline{D}(A) = \inf_{\mu \in M_G(X)}\mu(A) = \lim\limits_{\eps \to 0+}\underline{D}(A^{-\eps}) \quad \mbox{when $A$ is open}, \\
\label{UpperDensityFormula}
    \overline{D}(A) = \sup_{\mu \in M_G(X)}\mu(A) = \lim\limits_{\eps \to 0+}\overline{D}(A^{\eps}) \quad \mbox{when $A$ is closed}.
\end{gather}

Parts (i) and (ii) of the following definition were introduced in \cite{Ker20} and \cite{KerSza20} respectively. Parts (iii) and (iv) are their natural weaker versions, similar to the Definition 6.3 in \cite{DowZha17}.

\begin{definition}
\label{comp}
Let $G\curvearrowright X$ be an action of a countable group on a compact metrizable space and $m$ a nonnegative integer. We say that the action has
\begin{enumerate}[label=(\roman*)]
\item \textit{comparison} if $A \prec B$ for all open sets $A,B \subseteq X$ such that $\mu(A) < \mu(B)$ for every $G$-invariant measure $\mu$,
\item \textit{$m$-comparison} if $A \prec_m B$ for all open sets $A,B \subseteq X$ such that $\mu(A) < \mu(B)$ for every $G$-invariant measure $\mu$,
\item \textit{weak comparison} if $A \prec B$ for all open sets $A,B \subseteq X$ such that $\overline{D}(A) < \underline{D}(B)$,
\item \textit{weak $m$-comparison} if $A \prec_m B$ for all open sets $A,B \subseteq X$ such that $\overline{D}(A) < \underline{D}(B)$.
\end{enumerate}
\end{definition}

\begin{remark}
In the above definition one can equivalently take $A$ to range over closed sets
instead of open ones.
\end{remark}

For a free action on a zero-dimensional space Downarowicz and Zhang proved that comparison and weak comparison are equivalent \cite[Lemma 6.4]{DowZha17}. We modify their argument to prove the following.

\begin{lemma}
\label{ComparisonLemma}
Let $G\curvearrowright X$ be a minimal action of a countably infinite group on a compact metrizable space. Then the properties {\normalfont (i)-(iv)} are equivalent.
\end{lemma}

\begin{proof}
Note that any minimal actions on a finite space has comparison. We can therefore assume that $X$ is infinite. It is clear that (i) implies (ii) and (iii) and that either of them implies (iv). Thus, it remains to show that (iv) implies (i). Let $A$ be a closed set and $B$ an open set such that $\mu(A) < \mu(B)$ for every $G$-invariant measure $\mu$. It follows from \cite[Lemma 3.3]{KerSza20} that there exists an $\eps > 0$ such that $\mu(B) - \mu(A) > \eps$ for every $G$-invariant measure $\mu$. Choose a point $x \in X$ and elements $h_0, h_1, \ldots, h_m$ of the group such that $h_ix \in B$ for every $i = 0, 1, \ldots, m$ and $h_ix \not = h_jx$ unless $i=j$. Using \eqref{UpperDensityFormula} we can find a neighbourhood $U$ of $x$ such that $h_iU$ are disjoint subsets of $B$ for $i = 0, 1, \ldots, m$ and $\overline{D}(U) < \overline{D}(\overline{U}) < \frac{\eps}{(m+1)}$. Since the action is minimal, we have $\underline{D}(U) > 0$. Define $B_1 = B \setminus (h_0\overline{U} \cup h_1\overline{U} \cup \ldots \cup h_m\overline{U})$ and note that it is an open set such that $\mu(B_1) > \mu(A)$ for every $G$-invariant measure $\mu$.

Set $A_1 = A$ and use \cite[Lemma 3.3]{KerSza20} again to find $\eps_1 > 0$ such that $\mu(B_1^{-\eps_1}) > \mu(A_1^{\eps_1}) + \eps_1$ for every $G$-invariant measure $\mu$. Next, choose an arbitrary linear order for the elements of the group, i.e. write $G = \{g_1, g_2, \ldots\}$ and do the following inductive process. Assume that at step $n$ we have $\mu(B_n^{-\eps_n}) > \mu(A_n^{\eps_n}) + \eps_1$. Choose $\eps_{n+1}$ such that
\begin{gather*}
    \eps_{n+1} < \eps_n, \\
    g_n^{-1}\brak{\brak{g_nA_n}^{3\eps_{n+1}}} \subset A_n^{\eps_n}, \\
    g_n\brak{\brak{g_n^{-1}\brak{X\setminus B_n}}^{2\eps_{n+1}}} \subset \brak{X\setminus B_n}^{\eps_n},
\end{gather*}
and define 
\begin{gather*} 
    U_n = g_n^{-1}\left(B_n \cap \brak{g_nA_n}^{\eps_{n+1}}\right), \\
    B_{n+1} = B_n \setminus \overline{g_nU_n}, \\
    A_{n+1} = A_n \setminus U_n.
\end{gather*}
Similar to the inductive step in the proof of Lemma \ref{D-lemma} one can show that
$$g_n^{-1}\brak{B_n^{-\eps_n} \setminus B_{n+1}^{-\eps_{n+1}}} \subset A_n^{\eps_n}\setminus A_{n+1}^{\eps_{n+1}}.$$
It follows that for every $G$-invariant measure $\mu$
$$\mu\brak{B_{n+1}^{-\eps_{n+1}}} > \mu\brak{A_{n+1}^{\eps_{n+1}}} + \eps_1.$$

Set
$$A_\infty = \bigcap_{k=1}^\infty A_k, \quad B_\infty = \bigcap_{k=1}^\infty B_k.$$
Clearly, for every $G$-invariant measure $\mu$
$$\mu(B_\infty) - \mu(A_\infty) > \eps_1.$$

We claim that $\overline{D}(A_\infty) = 0$. By \eqref{UpperDensityFormula} it is sufficient to show that $\mu(A_\infty) = 0$ for every ergodic measure $\mu$. However, by construction we have $GA_\infty \cap B_\infty = \emptyset$. Since $\mu(B_\infty) > 0$ and $\mu$ is ergodic that can only happen if $\mu(A_\infty) = 0$.

Since each of the sets $A_k$ is closed, the function $\mu \mapsto \mu(A_k)$ is upper semicontinuous on the set of $G$-invariant measures. Thus, by Dini's theorem the sequence $\mu \mapsto \mu(A_k)$ converges uniformly to its pointwise limit $\mu \mapsto \mu(A_\infty) = 0$. Therefore, we can find $k$ such that $\overline{D}(A_k) < \underline{D}(U)$. By weak $m$-comparison, $A_k \prec_m U$ which implies that $A_k \prec h_0U \cup h_1U \cup \ldots \cup h_mU$. By construction $A \setminus A_k$ is covered by open sets $U_1, U_2, \ldots, U_{k-1}$ which become disjoint open subsets of $B_1$ after applying $g_1, g_2, \ldots, g_{k-1}$ respectively. We conclude that $A \prec B$.
\end{proof}




\section{Main results}
\label{ResultsSec}

\begin{lemma}
\label{D-lemma}
Let $G\curvearrowright X$ be an action of a discrete group on a metrizable space $X$. Suppose that $A \subset X$ is close, $B \subset X$ is open, and there exist a finite set $D \subset G$ and $\eps > 0$ such that for every $x \in X$
\begin{equation}
\label{D-densityIneq}
    \abs{\{g \in D^{-1}D \mid gx \in A^\eps\}} < m\abs{\{g \in D \mid gx \in B^{-\eps}\}}.
\end{equation}
Then $A \prec_{m-1} B$.
\end{lemma}
\begin{proof}
Choose a bijection $\sigma$ from the set $D \times \{1, \ldots, m\}$ to the set $\{1, \ldots, m\left|D\right|\}$ (i.e., endow the first set with a linear order). We will perform an algorithm with $m\left|D\right|$ steps which will end up producing families $\cU_{1}, \cU_{2}, \ldots, \cU_{m}$ of disjoint open sets and group elements $g_U$ for each set in every family which taken together implement the subequivalence $A \prec_{m-1} B$. The algorithm is similar to the one in the proof of Lemma \ref{ComparisonLemma}: it is greedy in nature and we use the given margin of error (initially $\eps$) to modify the sets so that they become either open or closed.

For the initial data at step 1 set $A_1=A$, $B_{1, k}=B$ for all $k = 1, \ldots, m$, $\eps_1 = \eps$ and $\cU_{k} = \emptyset$ for all $k = 1, \ldots, m$. Now, assume that at step $n$ we have the initial data consisting of a closed set $A_n$, a $m$-tuple of open sets $(B_{n, 1}, B_{n, 2}, \ldots, B_{n, m})$ and an $\eps_n > 0$ such that
\begin{equation}
\label{OrbitDensityEq}
    \abs{\{g \in D^{-1}D \mid gx \in A_n^{\eps_n}\}} < \sum\limits_{k=1}^{m} \abs{\{g \in D \mid gx \in B_{n, k}^{-\eps_n}\}}
\end{equation}
for all $x \in X$. It follows from \eqref{D-densityIneq} that the initial data for the first step satisfies those assumptions. Let $\sigma^{-1}(n) = (g, k_0)$ and choose $\eps_{n+1}$ such that
\begin{gather*}
    \eps_{n+1} < \eps_n, \\
    g^{-1}\brak{\brak{gA_n}^{3\eps_{n+1}}} \subset A_n^{\eps_n}, \\
    g\brak{\brak{g^{-1}\brak{X\setminus B_{n, k_0}}}^{2\eps_{n+1}}} \subset \brak{X\setminus B_{n, k_0}}^{\eps_n}.
\end{gather*}
Set $U_n = g^{-1}\left(B_{n, k_0} \cap \brak{gA_n}^{\eps_{n+1}}\right)$. Add the set $U_n$ to the family $\cU_{k_0}$ and then define 
\begin{gather*} 
    B_{n+1, k} = B_{n, k}, \quad \mbox{if} \quad  k \ne k_0; \\
    g_{U_n} = g, \quad B_{n+1, k_0} = B_{n, k_0} \setminus \overline{gU_n}, \quad A_{n+1} = A_n \setminus U_n.
\end{gather*}
We claim that this new data satisfies \eqref{OrbitDensityEq} and so the next step can be applied. Indeed, it is clear that on the right hand side of the equation the only part which might get smaller between steps $n$ and $n+1$ is $k_0$-th summand. Suppose $h \in D$ is an element such that $hx$ belongs to $B_{n, k_0}^{-\eps_n}$ but not to $B_{n+1, k_0}^{-\eps_{n+1}}$. The latter implies that $d(hx, (X \setminus B_{n, k_0}) \cup g\overline{U_n}) \le \eps_{n+1}$. However, from $hx \in B_{n, k_0}^{-\eps_n}$ it follows that $d(hx, X \setminus B_{n, k_0}) \ge \eps_n > \eps_{n+1}$ which means that $d(hx, g\overline{U_n}) \le \eps_{n+1}$. Since $gU_n$ is a subset of $\brak{gA_n}^{\eps_{n+1}}$ it follows that $hx$ is in $\brak{gA_n}^{3\eps_{n+1}}$ and therefore by our choice of $\eps_{n+1}$ the point $g^{-1}hx$ is in $A_n^{\eps_n}$. However, $A_{n+1} \subset g^{-1}(X \setminus B_{n, k_0})$ and since $hx \not \in \brak{X\setminus B_{n, k_0}}^{\eps_n}$ we obtain that $g^{-1}hx \not \in A_{n+1}^{\eps_{n+1}}$. Therefore, for every $h \in D$ such that $hx$ is in $B_{n, k_0}^{-\eps_n}$ but not in $B_{n+1, k_0}^{-\eps_{n+1}}$ we have that $g^{-1}h \in D^{-1}D$ and $g^{-1}hx$ is in $A_n^{\eps_n}$ but not in $A_{n+1}^{\eps_{n+1}}$. Thus, the inequality \eqref{OrbitDensityEq} holds for the data at step $n+1$. Note that the data produced in the last $m\abs{D}$-th step also satisfies \eqref{OrbitDensityEq}. 

It remains to prove that the families $\cU_{k}$, where $k = 1, 2, \ldots, m$, implement the subequivalence $A \prec_{m-1} B$. It is clear from the construction that the sets $g_{U_n}U_n$ are disjoint open subsets of $B$ when $U_n$ are taken from the same family $\cU_{k}$. We thus need only verify that the collection $U_1, U_2, \ldots, U_{m\abs{D}}$ covers $A$. If this is not the case then $A_{m\abs{D}+1}$ is non-empty and we can find a point $x$ in this set. For every pair $(g, k) \in D \times \{1, \ldots, m\}$ the point $x$ is not in $U_{\sigma(g, k)}$ which, in particular, implies that $gx$ is not in $B_{\sigma(g, k), k}$ and therefore not in $B_{m\abs{D}+1, k}$. Thus, $\abs{\{g \in D^{-1}D \mid gx \in A_{m\abs{D}+1}^{\eps_{m\abs{D}+1}}\}} \ge 1$ while for each $k$ the set $\{g \in D \mid gx \in B_{m\abs{D}+1, k}^{-\eps_{m\abs{D}+1}}\}$ is empty. This is in contradiction to \eqref{OrbitDensityEq}.
\end{proof}

\begin{theorem}
\label{MainThm}
Let $G$ be a finitely generated group with polynomial growth of order $\mathrm{ord}$ and $G\curvearrowright X$ be a continuous action on a compact metric space. Then the action has weak $\mathrm{2^{ord}}$-comparison.
\end{theorem}

\begin{proof}
Denote $\mathrm{2^{ord}}+1$ by $m$. Let $C$ be an open set and $A$ be a closed set such that $\overline{D}(A) < \underline{D}(C)$. Find $\eps > 0$ such that $\overline{D}(A^\eps) < \underline{D}(C^{-\eps})$ (see \eqref{LowerDensityFormula} and \eqref{UpperDensityFormula}). Equip $G$ with the word metric w.r.t. some symmetric generating set and let $B_n$ be a ball of radius $n$ around the identity. By the assumption that the growth is polynomial with order $\mathrm{ord}$ and the fact that the sequence $B_n$ is F{\o}lner (see, e.g. \cite{Tes07}) it follows that there is an $N$ such that $\abs{B_{2N}} < m\abs{B_N}$ and for all $n_1, n_2 \ge N$ and $x_1, x_2 \in X$ we have 
$$\frac{\left|\{g \in B_{n_1} \mid gx_1 \in A^\eps\}\right|}{\left|B_{n_1}\right|} < \frac{\left|\{g \in B_{n_2} \mid gx_2 \in C^{-\eps}\}\right|}{\left|B_{n_2}\right|}.$$
In particular, for all $x \in X$
\begin{equation*}
\left|\{g \in B_{2N} \mid gx \in A^\eps\}\right| < m \left|\{g \in B_{N} \mid gx \in C^{-\eps}\}\right|.
\end{equation*}
It now follows from Lemma \ref{D-lemma} that $A \prec_{m-1} B$.

\end{proof}

\begin{proof}[Proof of Theorems \ref{T-A} and \ref{T-B}]
Theorem \ref{T-A} follows from Theorem \ref{MainThm} and Lemma \ref{ComparisonLemma}. 

It was proved in \cite[Theorem 6.1]{KerSza20} that almost finiteness is equivalent to the small boundary property together with comparison. In the paper where he introduced almost finiteness Kerr proved that if a minimal action of a countably infinite amenable group is almost finite then the corresponding $C^*$-crossed product is $\cZ$-stable \cite[Theorem 12.4]{Ker20}. The UCT is satisfied automatically (\cite{Tu99}) and so it follows that such crossed products are classified by their Elliott invariant (see \cite{TikWhiWin17}, \cite{GonLinNiu15} and \cite{EllGonLinNiu15}). This completes the proof of Theorem \ref{T-B}.
\end{proof}

\begin{remark}
It is clear from the proof of \cite[Theorem 6.1]{KerSza20} given there that weak $m$-comparison is sufficient for the argument. Thus, for all free (not necessarily minimal) actions of finitely generated groups of polynomial growth on a compact metrizable space almost finiteness is equivalent to the small boundary property.
\end{remark}

\begin{remark}
There are free minimal actions of $\Zb$ which are not almost finite and such that the crossed product fails to have strict comparison (\cite{GioKer10} and \cite[Example 12.5]{Ker20}). Thus, we can conclude the following.
\begin{enumerate}
    \item The dynamical version of the Toms-Winter conjecture fails: comparison does not imply almost finiteness.
    \item Dynamical comparison does not imply strict comparison of the crossed product. 
\end{enumerate}
\end{remark}

\end{document}